\documentclass{amsart}
\usepackage[utf8]{inputenc}
\usepackage{tikz}
\usetikzlibrary{arrows.meta}
\usetikzlibrary{positioning}
\title{Coxeter groups and quiver representations}
\author{Hugh Thomas}
\address{D\'epartement de math\'ematiques, Universit\'e du Qu\'ebec \`a Montr\'eal, Montréal, Québec, Canada}
\email{hugh.ross.thomas@gmail.com}

\newcommand{\siaj}{(a:i\rightarrow j)}
\newcommand{\add}{\operatorname{add}}
\newcommand{\sjai}{(a:j\rightarrow i)}

\newcommand{\rep}{\operatorname{rep}}
\newcommand{\Ext}{\operatorname{Ext}}
\newcommand{\udim}{\operatorname{\underline {dim}}}
\newcommand{\la}{\langle}
\newcommand{\ra}{\rangle}
\newcommand{\Hom}{\operatorname{Hom}}
\newcommand{\Phir}{\Phi_{\it{re}}}
\newtheorem{theorem}{Theorem}[section]
\newtheorem{proposition}{Proposition}[section]
\newtheorem{lemma}{Lemma}[section]
\theoremstyle{definition}
\newtheorem{example}{Example}[section]
\newcommand{\inv}{\operatorname{inv}}
\begin{document}
\maketitle
In this expository note, I will attempt to showcase the relevance of Coxeter
groups
to quiver representations.  The subjects which I will discuss
are:\begin{enumerate}
\item real and imaginary roots,
\item reflection functors,
\item torsion free classes and $c$-sortable elements.
\end{enumerate}

The first two of these topics are classical, going back respectively to Kac \cite{Ka} and Bernstein-Gelfand-Ponomarev \cite{BGP}.
Let $Q$ be a quiver with $n$ vertices having no oriented cycles.  Associated to $Q$ is a Coxeter group
$W$ which acts on $\mathbb Z^n$.  
Real and imaginary roots in $\mathbb Z_{\geq 0}^n \subseteq \mathbb Z^n$
are defined using $W$, and characterize the dimension vectors of
the indecomposable
representations of $Q$.  This $W$-action on a set containing the dimension
vectors of representations of $Q$ suggests the idea
that $W$ might act directly on representations.  Such an action does not
exist. Reflection functors are, in a sense, the best we can do towards
defining such an action. 
For this classical material I will not give
proofs, but I will provide pointers to the textbooks \cite{DDPW} and \cite{K}.
These are particularly convenient for us since they take quiver representations as
their starting point.

The third topic which I will cover is a more
recent development. Torsion free classes of $\rep Q$ are full subcategories closed under extensions and subrepresentations.  
It is natural to want to classify them.  It turns out that the torsion free
classes that contain only finitely many isomorphism classes of indecomposables
correspond to $c$-sortable elements of $W$, a notion which was introduced by
Reading \cite{Re}.  This correspondence was first established in Dynkin type in
\cite{IT}.  The general result was shown by \cite{AIRT} as a consequence of
their study of preprojective algebras.  In this note, we will give a more
elementary proof, essentially following the argument of \cite{IT}, but without
restricting to the Dynkin case.

This note is based on an expository talk given at the Maurice Auslander international conference in 2013.  I am very grateful to the organizers for having
been invited to speak there, and for their patience in waiting for me to
write up the talk.  I am also happy to acknowledge financial support from
NSERC and the Canada Research Chairs program.  I thank the referees for their 
helpful comments which improved the paper.  

\section{Getting started}

Let $Q$ be a quiver without oriented cycles.  (It would
be possible to work over a general hereditary Artin algebra, but for
concreteness, I prefer not to.  For parts of this note, e.g. in Section 2,
one could also relax the prohibition on oriented cycles in $Q$.)
Suppose that $Q$ has $n$ vertices,
numbered 1 to $n$.  

We write $\rep Q$ for the representations of $Q$ over a fixed ground field
$k$.  Of our main references, 
note that \cite{DDPW} works over an arbitrary ground
field, while \cite{K} works only over an algebraically closed ground field.
Since for the topics we present, the choice of a general ground field does
not introduce any additional complications, we do not assume that $k$ is
algebraically closed.

Write $S_i$ for the
simple representation which has a copy of $k$ at vertex $i$ and the 0
vector space elsewhere.  

For $V\in \rep Q$, write $V_i$ for the vector space over vertex $i$.
Write $\udim V$ for the $n$-tuple $(\dim V_1,\dots,\dim V_n)\in
\mathbb Z^n$.

\subsection{The Euler-Ringel form}
There is an extremely useful bilinear form on $\mathbb Z^n$.

\begin{theorem}[{\cite[Theorem 1.25]{K}, \cite[Proposition 1.9(2)]{DDPW}}] The following equation defines a bilinear form on
  $\mathbb Z^n$, known as the Euler-Ringel form:
$$\la \udim V,\udim W\ra = \dim \Hom(V,W)-\dim \Ext^1(V,W)$$
\end{theorem}
  
It is not obvious that this defines a bilinear form!  
In order for it to do so, the righthand
side must
depend only on $\udim V$ and $\udim W$, and not on the specific
representations $V$ and $W$.  On the face of it, one would not expect this
to be true.
Indeed, the dimensions of
$\Hom(V,W)$ and of $\Ext^1(V,W)$ \emph{do} depend on the specific representations,
rather than on their dimension vectors.  However, taking the difference, we get a quantity which only depends on
$\udim V$ and $\udim W$.  

Once we know that this the righthand side only depends on $\udim V$, $\udim W$,
it is easy to see that it is linear, and to write down an explicit
formula for the pairing, because we can choose the most convenient
representations of the given dimension vectors.  If we want to
evaluate $\la \beta,\gamma\ra$,
the simplest choice is to take semisimple representations, i.e.,
representations which are direct sums of simple representations.  Since
$\Hom(S_i,S_j)$ is one-dimensional if $i=j$ and zero-dimensional otherwise,
and the dimension of $\Ext^1(S_i,S_j)$ is the number of arrows from $i$ to
$j$, we find that
$$\la \beta,\gamma\ra = \sum_{i=1}^n \beta_i\gamma_i -\sum_{\siaj\in Q}
\beta_i\gamma_j.$$

Note that this bilinear form is not symmetric.
We should view the fact that we are naturally given a non-symmetric
bilinear form as a piece of very good luck.  
Sometimes, though, it is useful to
have a symmetric bilinear form.  When we want one, it is easy to get one:
$$( \beta,\gamma)=\la \beta,\gamma\ra + \la \gamma,\beta\ra$$

If, on the contrary, we only had the symmetric form, we would not know how
to unsymmetrize it in a meaningful way.  (This is not a hypothetical situation!
I got into representation theory from Coxeter groups, and, working on Coxeter
groups, I became convinced that I needed to break the symmetry of the
symmetric bilinear form, but, not knowing anything about quiver
representations, I had no idea how to do it.)

\subsection{The Weyl group associated to $Q$}
For $\beta\in \mathbb Z^n$ with $(\beta,\beta)\ne 0$, we can define a reflection
on $\mathbb Q^n$, which we denote $t_\beta$, by:
$$t_\beta(\gamma)=\gamma-2\frac{(\beta,\gamma)}{(\beta,\beta)}\beta$$
This preserves the symmetric form, fixes all vectors orthogonal to $\beta$,
and sends $\beta$ to
$-\beta$.  If $|(\beta,\beta)|$ divides 2, then $t_\beta$ restricts to a map from
$\mathbb Z^n$ to $\mathbb Z^n$.  

It is convenient to let $e_i$ denote the standard basis vector consisting of
all zeros, with a 1 in the $i$-th position.  Equivalently, we can say 
$e_i=\udim S_i$.  

Let us define $s_i=t_{e_i}$.  Note that, by our definition of the symmetric
form, $(e_i,e_i)=2$, so $s_i=t_{e_i}$ is well-defined as a map from $\mathbb Z^n$
to $\mathbb Z^n$.  We refer to the elements $s_i$ as \emph{simple reflections}.

Define $W$ to be the group generated by the elements $s_i$.  Since this is a
group generated by reflections, it is what is called a \emph{reflection
  group}.  The generators satisfy some relations which it is easy to verify:
\begin{itemize}
\item $s_i^2=e$,
  \item if $i$ and $j$ are not adjacent, then $s_is_j=s_js_i$,
  \item if there is exactly one arrow between $i$ and $j$, then
    $s_is_js_i=s_js_is_j$,
  \item for clarity, let me add that if there is more than one arrow between
    $i$ and $j$, then
    there is no corresponding relation.
\end{itemize}
It is natural to wonder if there are any other relations.  It turns out that
there are not.

\begin{theorem}[{\cite[Theorem 4.2(1)]{DDPW}, \cite[Theorem A.11]{K}}] The group $W$ can be defined abstractly as the group generated by
  $n$ elements $s_1,\dots,s_n$ subject to the above relations. \end{theorem}

The group $W$ is therefore a Coxeter group.  For background on Coxeter
groups, see \cite{BB,H}.

Let us remark at this point that the group $W$ does not depend on the
orientation of the arrows of the quiver; that is to say, it depends
only on the underlying graph structure.  This is in marked contrast to
$\rep Q$, which depends highly on the orientation.  Nonetheless, as we shall
see, there are connections between the representations of different
orientations of the same quiver.  

\begin{example}
An example to bear in mind always is the case that the quiver is 
a path with the vertices numbered 1 to $n$ in order along it, and
the arrows oriented arbitrarily.  We refer to such a quiver as being of
type $A_n$.  The corresponding Weyl group is isomorphic to the symmetric
group on $n+1$ letters; we can take the isomorphism to send $s_i$ to
the adjacent transposition $(i\ i+1)$.  In examples, we represent an element
$w$ of the symmetric group in \emph{one-line notation}, by writing down
the sequence $w(1),\dots,w(n+1)$. \end{example}

By definition, each element of $W$ can be expressed as a product of
the simple reflections.  Such an expression for $w$ of minimal length is called
\emph{reduced}.
We write $\ell(w)$ for the length of a reduced expression
for $w$.  
Typically an element $w$ will have many reduced
expressions.  

For $w\in W$ and $s\in S$, if $w$ has a reduced expression beginning
with $s$, then it is easy to see that $\ell(sw)=\ell(w)-1$.  It turns out
that if $w$ does not have a reduced expression beginning with $s$, then
$\ell(sw)=\ell(w)+1$.  (When I write ``beginning with,'' I mean with
respect to the usual left-to-right order in which English is read.)

Define $\Phir = \{we_i \mid w\in W, 1\leq i\leq n\}$, and
$\Phir^+=\Phir \cap \mathbb Z_{\geq 0}^n$. The elements of $\Phir^+$ are called the
real roots for $Q$.  

\begin{example} For type $A_{n}$, we find that $\Phir^+$ consists of the sums $e_i+ e_{i+1}+\dots+e_j$ for $1\leq i\leq j\leq n$.  \end{example}

Let $s_{i_1}\dots s_{i_r}$ be a reduced expression for $w\in W$.  Define the
inversion set corresponding to the reduced expression to be the following set of real roots: $\{e_{i_1},s_{i_1}e_{i_2},s_{i_1}s_{i_2}e_{i_3},\dots,s_{i_1}s_{i_2}\dots s_{i_{r-1}}e_{i_r}\}$.  This can also be viewed as the collection of positive roots which are sent to negative roots by $w^{-1}$.  Note that this latter description
does not depend on the choice of reduced expression for $w$; indeed, we
have the following proposition:

\begin{proposition}[{\cite[Corollary 1.4.4, Corollary 1.4.5, and Corollary 3.1.4]{BB}}]\label{same-diff} For $w\in W$, all the reduced words for $w$ yield the
  same inversion set.  Different elements of $W$ have different inversion sets.
The size of the inversion set of $w$ is $\ell(w)$. \end{proposition}

We write $\inv(w)$ for the inversion set of $w$.  

\begin{example} In type $A_n$, for $j\geq i$, we have $e_i+e_{i+1}+\dots+e_j \in \inv(w)$ if and
  only if $j+1$ precedes $i$ in the one-line notation for $w$.
  
  In type $A_2$, the six elements of $W$, their expressions in one-line notation, and their corresponding
  inversion sets, are as follows:
  $$\begin{array}{ccc}
    e & 123 & \emptyset\\
    s_1 & 213& \{e_1\}\\
    s_2 & 132 & \{e_2\}\\
    s_1s_2& 231 &\{e_1,e_1+e_2\}\\
    s_2s_1& 312 & \{e_2,e_1+e_2\}\\
    s_1s_2s_1& 321 & \{e_1,e_1+e_2,e_2\}\end{array}$$
\end{example}
We have listed the inversion sets in the order corresponding to the reduced
expressions given above.  If we take a different reduced word for the same element, considering $s_2s_1s_2$ instead of $s_1s_2s_1$, we get
the roots $e_2,e_1+e_2,e_1$, in that order.  Consistent with Proposition \ref{same-diff}, the set of roots is the same as appears in the table for $s_1s_2s_1$, though the order in which they appear is different.  

\subsection{Finite versus infinite} There is a basic dichotomy in root
system combinatorics.  In our setting, it manifests in the following
theorem:

\begin{theorem} The following are equivalent:
  \begin{itemize}
  \item $W$ is finite,
  \item $\rep Q$ has only finitely many indecomposables up to isomorphism (in which case we say $Q$ is of \emph{finite representation type}),
  \item $Q$ consists of a finite number of Dynkin quivers, i.e., quivers of type $A_n$ ($n\geq 1$), $D_n$ ($n\geq 4$), $E_6$, $E_7$,
    or $E_8$.
  \end{itemize}
\end{theorem}

  We have already met quivers of type $A_n$.  A quiver of type $D_n$ is a
  quiver with $n$ vertices, obtained by orienting the following graph:

  $$
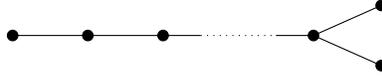
\begin{tikzpicture}
    \draw (0,0) -- (2.5,0);
    \draw[dotted] (2.5,0) -- (3.5,0);
    \draw (3.5,0) -- (4,0);
    \draw (4,0) -- (4.9,.4);
    \draw (4,0) -- (4.9,-.4);
    \filldraw [black] (0,0) circle (2pt) (1,0) circle (2pt) (2,0) circle (2pt)
    (4,0) circle (2pt) (4.9,.4) circle (2pt) (4.9,-.4) circle (2pt);
  \end{tikzpicture}$$

  The quivers of type $E_6$, $E_7$, and $E_8$ are obtained by orienting the
  graphs shown below:

  $$
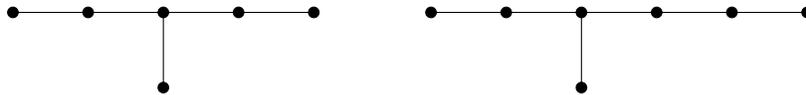
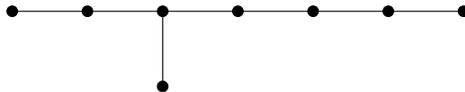
\begin{tikzpicture}
    \draw (0,0) -- (4,0);
    \draw (2,0) -- (2,-1);
    \filldraw[black] (0,0) circle (2pt) (1,0) circle (2pt)
    (2,0) circle (2pt) (3,0) circle (2pt) (4,0) circle (2pt)
    (2,-1) circle (2pt);
    \end{tikzpicture}
\quad\quad\quad\quad
  \begin{tikzpicture}
    \draw (0,0) -- (5,0);
    \draw (2,0) -- (2,-1);
    \filldraw[black] (0,0) circle (2pt) (1,0) circle (2pt)
    (2,0) circle (2pt) (3,0) circle (2pt) (4,0) circle (2pt)
    (2,-1) circle (2pt) (5,0) circle (2pt);
    \end{tikzpicture}$$
$$\begin{tikzpicture}
    \draw (0,0) -- (6,0);
    \draw (2,0) -- (2,-1);
    \filldraw[black] (0,0) circle (2pt) (1,0) circle (2pt)
    (2,0) circle (2pt) (3,0) circle (2pt) (4,0) circle (2pt)
    (2,-1) circle (2pt) (5,0) circle (2pt) (6,0) circle (2pt);
    \useasboundingbox(0,.3);
    \end{tikzpicture}$$
  
The equivalence of (ii) and (iii) is the celebrated theorem of Gabriel \cite[Theorem 1.23]{DDPW}, \cite[Theorem 3.3]{K}.
The equivalence of (i) and (iii) follow from the classification of finite
Coxeter groups, see for example \cite[Theorem A.11]{K}. 

\section{Real and imaginary roots} 

It turns out that $\Phir^+$, as defined above, is intimately connected to
the representation theory of $Q$.  

\begin{theorem}[{\cite[Theorem 7.49]{K}, \cite[Remark 1.26]{DDPW}}] \label{real-roots}
  There is a unique indecomposable representation of $Q$ whose dimension vector
  is $\alpha$ for each $\alpha\in\Phir^+$.
  \end{theorem}

This theorem was established under the assumption that $k$ is algebraically
closed by Kac \cite{Ka}.  As pointed out by Schofield \cite{Sch}, Kac's proof also applies to any field of characteristic $p$.  The characteristic zero case
is proved in \cite{Sch}.  Note that exceptionally \cite{DDPW} only states the result for
algebraically closed ground fields.

If $Q$ is of finite representation
type, then the representations of Theorem \ref{real-roots}
exhaust the indecomposables.  However, if $Q$ is not of finite
type, there is another class of roots, the \emph{imaginary roots}.

Let $M$ be the cone in $\mathbb Z^n_{\geq 0}$ consisting of vectors with connected
support and such that for $\alpha\in M$, we have $(\alpha,e_i)\leq 0$.
The imaginary roots are all elements of $\mathbb Z^n$ of the form
$w\alpha$ for $w \in W$ and $\alpha \in M$.  The following result is due
to Kac \cite{Ka}:

\begin{theorem}[{\cite[Theorem 7.49 and Theorem A.20]{K}}]
  Let $k$ be algebraically closed.  The dimension vectors of
  indecomposable representations of $Q$ are the positive real roots together
  with the imaginary roots.  The positive real roots are those dimension
  vectors for which there
  exists a unique indecomposable representation, while the 
  imaginary roots are those
  dimension vectors for which there are infinitely many non-isomorphic
  indecomposable representations.  
\end{theorem}

This theorem obviously requires some modification to accommodate finite fields; in 
fact, 
a full analogue of this theorem for general fields is not known.  

\section{Reflection functors}
In this section we define the reflections functors, which play a decisive role
in the representation theory of quivers of finite representation type, but
which, as we will see, can also be important for general quiver representations.
The recognition of their importance goes back to the classic paper of
Bernstein-Gelfand-Ponomarev \cite{BGP}.

We say that a vertex $i$ is a \emph{sink} of $Q$ if all the arrows incident to $i$
point towards $i$.  Dually, we say $i$ is a \emph{source} if all arrows incident
to $i$ point away from it.

If $i$ is a sink or a source, we define $\mu_i(Q)$ to be the quiver
obtained by changing the direction of all the arrows incident to $i$.

Suppose that $i$ is a sink of $Q$.  
Let $V$ be a representation of $Q$.  We want to define a representation
$R_i^+(V)$ of $Q'=\mu_i(Q)$.

First, we define the vector spaces:

\begin{eqnarray*}
  R_i^+(V)_j = \begin{cases}
  V_j &\textrm { for $j\ne i$} \\
   \ker\left( \left.\bigoplus\limits_{\sjai \in Q} \!\!\!V_j\right) \rightarrow V_i\right. &\textrm { for $j=i$}
\end{cases}
  \end{eqnarray*}
Note that
the direct sum runs over all arrows pointing to $i$ in $Q$.  In the case that there
is more than one arrow from $j$ to $i$, the summand $V_j$ will be repeated.

We also have to define the linear maps associated to the arrows of
$\mu_i(Q)$.  For the arrows which are also in $Q$, we notice that neither
the source nor the target vector spaces of the arrow have changed, so we can
and do simply reuse the map from $V$.  It remains to define the maps
corresponding to the arrows in $\mu_i(Q)$ which point away from $i$.  We
use the maps coming from the natural
inclusion of $R_i^+(V)_i$ into the sum over all
the arrows from $j$ to $i$ in $Q$ of $V_j$.  

\begin{example} Let $Q$ be the quiver of type $A_3$ shown below.  Then
  $Q'=\mu_2(Q)$ is also as shown.
  $$
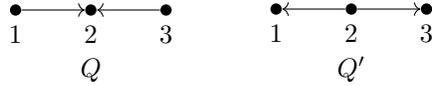
\begin{tikzpicture}
    \node (a) at (0,0) [shape=circle,fill=black,inner sep =0pt, minimum size=1.5mm] {};
    \node[below] (e) at (a.south) {$1$} ;
    \node (b) at (1,0) [shape=circle,fill=black,inner sep =0pt, minimum size=1.5mm] {};
    \node[below] (f) at (b.south) {$2$};
\node (c) at (2,0) [shape=circle,fill=black,inner sep =0pt, minimum size=1.5mm] {};

\node (d) at (1,-.8) {$Q$};
\node[below] (g) at (c.south) {$3$};
    \draw[->] (a) -- (b);
    \draw[->] (c) -- (b);
  \end{tikzpicture}
\quad\quad\quad
\begin{tikzpicture}
    \node (a) at (0,0) [shape=circle,fill=black,inner sep =0pt, minimum size=1.5mm] {};
    \node[below] (e) at (a.south) {$1$} ;
    \node (b) at (1,0) [shape=circle,fill=black,inner sep =0pt, minimum size=1.5mm] {};
\node (c) at (2,0) [shape=circle,fill=black,inner sep =0pt, minimum size=1.5mm] {};
    \node[below] (f) at (b.south) {$2$};
    \node[below] (g) at (c.south) {$3$} ;

    \node (d) at (1,-.8) {$Q'$};
    \draw[<-] (a) -- (b);
    \draw[<-] (c) -- (b);
  \end{tikzpicture}
$$

Below, in the lefthand column, we list the six indecomposable
representations of $Q$, and in the righthand column, the 
representation of $Q'$ to which it is sent by $R_2^+$ (which is indecomposable
or, in the first case, the zero representation).  We only give the vector
spaces at each vertex; the maps between consecutive copies of $k$ are
non-zero.
$$
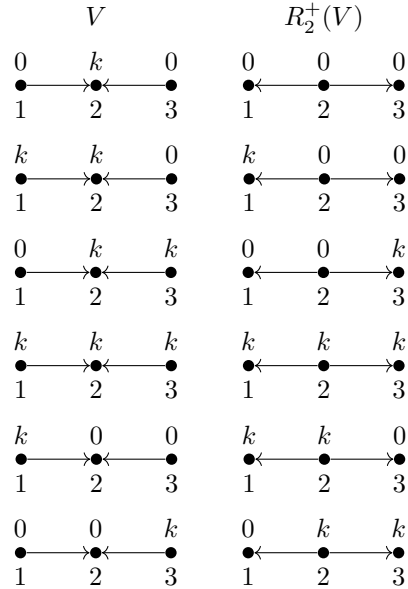
\arraycolsep=3mm
\begin{array}{cc}
  V & R_2^+(V)\\[1mm]
\begin{tikzpicture}
    \node (a) at (0,0) [shape=circle,fill=black,inner sep =0pt, minimum size=1.5mm] {};
    \node[below] (e) at (a.south) {$1$} ;
    \node (b) at (1,0) [shape=circle,fill=black,inner sep =0pt, minimum size=1.5mm] {};
    \node[below] (f) at (b.south) {$2$};
\node (c) at (2,0) [shape=circle,fill=black,inner sep =0pt, minimum size=1.5mm] {};
\node[below] (g) at (c.south) {$3$};
    \draw[->] (a) -- (b);
    \draw[->] (c) -- (b);
    \node[above] at (a.north) {$0$};
    \node[above] at (b.north) {$k$};
    \node[above] at (c.north) {$0$};
  \end{tikzpicture}
  &
  \begin{tikzpicture}
    \node (a) at (0,0) [shape=circle,fill=black,inner sep =0pt, minimum size=1.5mm] {};
    \node[below] (e) at (a.south) {$1$} ;
    \node (b) at (1,0) [shape=circle,fill=black,inner sep =0pt, minimum size=1.5mm] {};
\node (c) at (2,0) [shape=circle,fill=black,inner sep =0pt, minimum size=1.5mm] {};
    \node[below] (f) at (b.south) {$2$};
    \node[below] (g) at (c.south) {$3$} ;
    \draw[<-] (a) -- (b);
    \draw[<-] (c) -- (b);
        \node[above] at (a.north) {$0$};
    \node[above] at (b.north) {$0$};
    \node[above] at (c.north) {$0$};\end{tikzpicture}\\
  \begin{tikzpicture}
    \node (a) at (0,0) [shape=circle,fill=black,inner sep =0pt, minimum size=1.5mm] {};
    \node[below] (e) at (a.south) {$1$} ;
    \node (b) at (1,0) [shape=circle,fill=black,inner sep =0pt, minimum size=1.5mm] {};
    \node[below] (f) at (b.south) {$2$};
\node (c) at (2,0) [shape=circle,fill=black,inner sep =0pt, minimum size=1.5mm] {};
\node[below] (g) at (c.south) {$3$};
    \draw[->] (a) -- (b);
    \draw[->] (c) -- (b);
    \node[above] at (a.north) {$k$};
    \node[above] at (b.north) {$k$};
    \node[above] at (c.north) {$0$};
  \end{tikzpicture}
  &
  \begin{tikzpicture}
    \node (a) at (0,0) [shape=circle,fill=black,inner sep =0pt, minimum size=1.5mm] {};
    \node[below] (e) at (a.south) {$1$} ;
    \node (b) at (1,0) [shape=circle,fill=black,inner sep =0pt, minimum size=1.5mm] {};
\node (c) at (2,0) [shape=circle,fill=black,inner sep =0pt, minimum size=1.5mm] {};
    \node[below] (f) at (b.south) {$2$};
    \node[below] (g) at (c.south) {$3$} ;
    \draw[<-] (a) -- (b);
    \draw[<-] (c) -- (b);
        \node[above] at (a.north) {$k$};
    \node[above] at (b.north) {$0$};
    \node[above] at (c.north) {$0$};\end{tikzpicture}\\
  \begin{tikzpicture}
    \node (a) at (0,0) [shape=circle,fill=black,inner sep =0pt, minimum size=1.5mm] {};
    \node[below] (e) at (a.south) {$1$} ;
    \node (b) at (1,0) [shape=circle,fill=black,inner sep =0pt, minimum size=1.5mm] {};
    \node[below] (f) at (b.south) {$2$};
\node (c) at (2,0) [shape=circle,fill=black,inner sep =0pt, minimum size=1.5mm] {};
\node[below] (g) at (c.south) {$3$};
    \draw[->] (a) -- (b);
    \draw[->] (c) -- (b);
    \node[above] at (a.north) {$0$};
    \node[above] at (b.north) {$k$};
    \node[above] at (c.north) {$k$};
  \end{tikzpicture}
  &
  \begin{tikzpicture}
    \node (a) at (0,0) [shape=circle,fill=black,inner sep =0pt, minimum size=1.5mm] {};
    \node[below] (e) at (a.south) {$1$} ;
    \node (b) at (1,0) [shape=circle,fill=black,inner sep =0pt, minimum size=1.5mm] {};
\node (c) at (2,0) [shape=circle,fill=black,inner sep =0pt, minimum size=1.5mm] {};
    \node[below] (f) at (b.south) {$2$};
    \node[below] (g) at (c.south) {$3$} ;
    \draw[<-] (a) -- (b);
    \draw[<-] (c) -- (b);
        \node[above] at (a.north) {$0$};
    \node[above] at (b.north) {$0$};
    \node[above] at (c.north) {$k$}; \end{tikzpicture}\\
  \begin{tikzpicture}
    \node (a) at (0,0) [shape=circle,fill=black,inner sep =0pt, minimum size=1.5mm] {};
    \node[below] (e) at (a.south) {$1$} ;
    \node (b) at (1,0) [shape=circle,fill=black,inner sep =0pt, minimum size=1.5mm] {};
    \node[below] (f) at (b.south) {$2$};
\node (c) at (2,0) [shape=circle,fill=black,inner sep =0pt, minimum size=1.5mm] {};
\node[below] (g) at (c.south) {$3$};
    \draw[->] (a) -- (b);
    \draw[->] (c) -- (b);
    \node[above] at (a.north) {$k$};
    \node[above] at (b.north) {$k$};
    \node[above] at (c.north) {$k$};
  \end{tikzpicture}
  &
  \begin{tikzpicture}
    \node (a) at (0,0) [shape=circle,fill=black,inner sep =0pt, minimum size=1.5mm] {};
    \node[below] (e) at (a.south) {$1$} ;
    \node (b) at (1,0) [shape=circle,fill=black,inner sep =0pt, minimum size=1.5mm] {};
\node (c) at (2,0) [shape=circle,fill=black,inner sep =0pt, minimum size=1.5mm] {};
    \node[below] (f) at (b.south) {$2$};
    \node[below] (g) at (c.south) {$3$} ;
    \draw[<-] (a) -- (b);
    \draw[<-] (c) -- (b);
        \node[above] at (a.north) {$k$};
    \node[above] at (b.north) {$k$};
    \node[above] at (c.north) {$k$}; \end{tikzpicture}\\
  \begin{tikzpicture}
    \node (a) at (0,0) [shape=circle,fill=black,inner sep =0pt, minimum size=1.5mm] {};
    \node[below] (e) at (a.south) {$1$} ;
    \node (b) at (1,0) [shape=circle,fill=black,inner sep =0pt, minimum size=1.5mm] {};
    \node[below] (f) at (b.south) {$2$};
\node (c) at (2,0) [shape=circle,fill=black,inner sep =0pt, minimum size=1.5mm] {};
\node[below] (g) at (c.south) {$3$};
    \draw[->] (a) -- (b);
    \draw[->] (c) -- (b);
    \node[above] at (a.north) {$k$};
    \node[above] at (b.north) {$0$};
    \node[above] at (c.north) {$0$};
  \end{tikzpicture}
  &
  \begin{tikzpicture}
    \node (a) at (0,0) [shape=circle,fill=black,inner sep =0pt, minimum size=1.5mm] {};
    \node[below] (e) at (a.south) {$1$} ;
    \node (b) at (1,0) [shape=circle,fill=black,inner sep =0pt, minimum size=1.5mm] {};
\node (c) at (2,0) [shape=circle,fill=black,inner sep =0pt, minimum size=1.5mm] {};
    \node[below] (f) at (b.south) {$2$};
    \node[below] (g) at (c.south) {$3$} ;
    \draw[<-] (a) -- (b);
    \draw[<-] (c) -- (b);
        \node[above] at (a.north) {$k$};
    \node[above] at (b.north) {$k$};
    \node[above] at (c.north) {$0$}; \end{tikzpicture}\\

   \begin{tikzpicture}
    \node (a) at (0,0) [shape=circle,fill=black,inner sep =0pt, minimum size=1.5mm] {};
    \node[below] (e) at (a.south) {$1$} ;
    \node (b) at (1,0) [shape=circle,fill=black,inner sep =0pt, minimum size=1.5mm] {};
    \node[below] (f) at (b.south) {$2$};
\node (c) at (2,0) [shape=circle,fill=black,inner sep =0pt, minimum size=1.5mm] {};
\node[below] (g) at (c.south) {$3$};
    \draw[->] (a) -- (b);
    \draw[->] (c) -- (b);
    \node[above] at (a.north) {$0$};
    \node[above] at (b.north) {$0$};
    \node[above] at (c.north) {$k$};
  \end{tikzpicture}
  &
  \begin{tikzpicture}
    \node (a) at (0,0) [shape=circle,fill=black,inner sep =0pt, minimum size=1.5mm] {};
    \node[below] (e) at (a.south) {$1$} ;
    \node (b) at (1,0) [shape=circle,fill=black,inner sep =0pt, minimum size=1.5mm] {};
\node (c) at (2,0) [shape=circle,fill=black,inner sep =0pt, minimum size=1.5mm] {};
    \node[below] (f) at (b.south) {$2$};
    \node[below] (g) at (c.south) {$3$} ;
    \draw[<-] (a) -- (b);
    \draw[<-] (c) -- (b);
        \node[above] at (a.north) {$0$};
    \node[above] at (b.north) {$k$};
    \node[above] at (c.north) {$k$}; \end{tikzpicture}\\
\end{array}   
$$
\end{example}
  
In order for $R_i^+$ to be a functor, for $f\in \Hom(V,W)$, we must define
$R_i^+(f)\in \Hom(R_i^+(V),R_i^+(W))$.  For $j\ne i$, we define $R_i^+(f)_j=f_j$.
To define $R_i^+(f)_i$, we consider the following commutative diagram:

$$\begin{tikzpicture}[xscale=2,yscale=1.5]
  \node (a) at (0.3,0) {$0$};
  \node (b) at (1,0) {$R_i^+(V)_i$};
  \node (c) at (2,0) {\raisebox{-6mm}{$\bigoplus\limits_{\sjai \in Q}V_j$}};
  \node (d) at (3,0) {$V_i$};
  \node (xa) at (0.3,-1) {$0$};
  \node (xb) at (1,-1) {$R_i^+(W)_i$};
  \node (xc) at (2,-1) {\raisebox{-6mm}{$\bigoplus\limits_{\sjai \in Q}W_j$}};
  \node (xd) at (3,-1) {$W_i$};
  \draw[->] (a) -- (b);
  \draw[->] (b) -- (c);
  \draw[->] (c) -- (d);
  \draw[->] (xa) -- (xb);
  \draw[->] (xb) -- (xc);
  \draw[->] (xc) -- (xd);
  \draw[->,dotted] (b) -- (xb);
  \path[->] (c) edge ([yshift=-2mm] xc.north);
  \draw[->] (d) -- (xd);
  \end{tikzpicture}$$

The rows of the diagram are the exact sequences which define
$R_i^+(V)_i$ and $R_i^+(W)_i$.  The vertical undotted arrows are maps defined
by $f$; the
dotted arrow is the arrow we want to define.  We can see in the diagram a map from $R_i^+(V)_i$ to $\bigoplus_{\sjai} W_j$.  We observe that if we
compose it with the map from $\bigoplus_{\sjai} W_j$ to $W_i$, the result will be
zero, which means that in fact the image of $R_i^+(V)_i$ falls in
$R_i^+(W)_i$.  We take this map from $R_i^+(V)_i$ to $R_i^+(W)_i$ to be the definition
of $R_i^+(f)_i$.  

We have gone into considerable detail describing $R_i^+$; the definition of
$R_i^-$ which is defined when $i$ is a source is dual in the natural way.
Let me
mention only the key point that $R_i^-(V)_i$ is defined to be the cokernel
of the natural map from $V_i$ to $\bigoplus_{\siaj} V_j$.  
We then have the following theorem:

\begin{theorem}[{\cite[Theorem 1.18]{DDPW}, \cite[Theorem 3.10]{K}}]
  Let $i$ be a sink in $Q$, and let $Q'=\mu_i(Q)$.  Then
    $R^+_i$ is a functor from representations of $Q$ to representations
    of $Q'$, and $R^-_i$ is a functor from representations of $Q'$ to
    representations of $Q$. \end{theorem}

Let us note some important properties of $R_i^+$ and $R_i^-$.
\begin{proposition} Let $i$ be a sink in $Q$, and let $Q'=\mu_i(Q)$.
  We write $S_i$ for the simple representation at $i$ of $Q$, and
  $S_i'$ for the simple representation at $i$ of $Q'$.  
  \begin{enumerate}
\item  $R_i^+(S_i)=0$ and $R_i^-(S_i')=0$.
\item If $V\in \rep Q$ has no direct summand isomorphic to $S_i$, then
  $$\udim R_i^+(V)= s_i(\udim(V)).$$  Similarly if $V'\in \rep Q'$ has no
  direct summand isomorphic to $S_i'$ then $$\udim R_i^-(V')=s_i(\udim(V')).$$  
\item $R_i^+$ is left exact; $R_i^-$ is right exact.  
\item If $V\in \rep Q$, then $R_i^-(R_i^+(V))$ is isomorphic to the direct sum of the indecomposable summands of $V$ which are not isomorphic to $S_i$ (and
  a similar dual statement).
\item The image of $R_i^+(\rep Q)$ consists of all representations of $Q'$
  not having $S_i'$ as an indecomposable summand (and a similar dual statement).  
  \end{enumerate}
\end{proposition}
\begin{proof} 
  (1) is immediate from the definition.

  For (2), we prove on the first statement (about $R_i^+$).
  Suppose that $V$ has no direct summands isomorphic to
    $S_i$.  In this case, the map $\bigoplus_{\sjai} \!\!V_j \rightarrow V_i$
    is surjective.  This means that
    $$\dim R_i^+(V)_i=\left(\sum_{\sjai\in Q} \dim V_j\right) - \dim V_i,$$
    while $\dim R_i^+(V)_j=\dim V_j$ for $j\ne i$.  This claim follows.
    Property (2) explains both the name (``reflection functor'') and one of the
    reasons that this topic is included in this note.

    For (3), the statement is clear except over vertex $i$.  There, it follows
    from an application of the Snake Lemma.  For another proof, see
    \cite[Theorem 3.10]{K} or \cite[Theorem 1.18]{DDPW}.

    (4) follows from applying (2) to the decomposition of $V$ into
    indecomposables.

    For (5), we check directly that for $M\in \rep Q$, we have that
    $R_i^+(M)$ has no $S_i'$ direct summand.  The fact that all representations
    without such a summand are in the image follows from (4).  
  \end{proof}

  \medskip  
At the beginning of this section,
we defined $\mu_i(Q)$ to be the result of reversing the arrows
incident to $i$, provided that $i$ was either a sink or a source.  Clearly,
this definition could still be applied if $i$ is neither a sink nor a source, but
it turns out that, from the point of view of representation theory,
it is not the right thing to do: the right thing is to apply
(Fomin-Zelevinsky) mutation, but the whole story becomes more complicated.
See \cite{DWZ1} for more on this.  

\section{Torsion free classes and $c$-sortable elements}

We are now ready to introduce $c$-sortable elements of $W$ and finite torsion
free classes in $\rep Q$.  We will see that they have very similar
inductive structures, which we will exploit to construct a bijection between
them.

\subsection{The correspondence between Coxeter elements and orientations of a quiver}
Let $|Q|$ be the unoriented graph underlying the quiver $Q$.  As was already mentioned,
the Weyl group associated to $Q$ only depends on $|Q|$.  A Coxeter element for $W$ is,
by definition, a product of each of the Coxeter generators of $W$ in some order.  The 
following fact is easy but important.  

\begin{lemma} There
is a bijection between the acyclic orientations of $|Q|$ and the Coxeter elements of $W$.
\end{lemma}
\begin{proof} 
Associate to $Q$ the product of the Coxeter generators in which $s_i$ precedes $s_j$ if 
there are one or more arrows from $j$ to $i$ in the quiver.  This does not prescribe the
order of all pairs of Coxeter generators, but any pair whose order is not prescribed is
a pair whose corresponding vertices are not connected by an arrow, and which therefore commute.  The fact that it is possible to find a linear order on the generators satisfying
this prescription follows from the fact that the orientation of $Q$ has no cycles.  

Conversely, any Coxeter element determines an acyclic orientation of $|Q|$ by orienting
any edges between $i$ and $j$ towards $i$ if $s_i$ precedes $s_j$ in the Coxeter element.
\end{proof}

We write $c_Q$ for the Coxeter element associated to $Q$.  

\begin{example}
  Let $Q_1$ and $Q_2$ be as shown below:
  $$
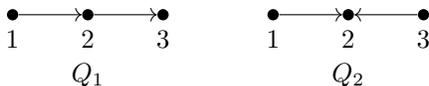
\begin{tikzpicture}
    \node (a) at (0,0) [shape=circle,fill=black,inner sep =0pt, minimum size=1.5mm] {};
    \node[below] (e) at (a.south) {$1$} ;
    \node (b) at (1,0) [shape=circle,fill=black,inner sep =0pt, minimum size=1.5mm] {};
    \node[below] (f) at (b.south) {$2$};
\node (c) at (2,0) [shape=circle,fill=black,inner sep =0pt, minimum size=1.5mm] {};

\node (d) at (1,-.8) {$Q_1$};
\node[below] (g) at (c.south) {$3$};
    \draw[->] (a) -- (b);
    \draw[->] (b) -- (c);
  \end{tikzpicture}
\quad\quad\quad
\begin{tikzpicture}
    \node (a) at (0,0) [shape=circle,fill=black,inner sep =0pt, minimum size=1.5mm] {};
    \node[below] (e) at (a.south) {$1$} ;
    \node (b) at (1,0) [shape=circle,fill=black,inner sep =0pt, minimum size=1.5mm] {};
\node (c) at (2,0) [shape=circle,fill=black,inner sep =0pt, minimum size=1.5mm] {};
    \node[below] (f) at (b.south) {$2$};
    \node[below] (g) at (c.south) {$3$} ;

    \node (d) at (1,-.8) {$Q_2$};
    \draw[->] (a) -- (b);
    \draw[->] (c) -- (b);
  \end{tikzpicture}
$$
In the Coxeter element associated to $Q_1$, we have $s_3$ before
$s_2$ and $s_2$ before $s_1$, so $c_{Q_1}=s_3s_2s_1$.  In the Coxeter element
associated to $Q_2$, we have $s_2$ before $s_1$ and $s_2$ before $s_3$.
We therefore could write $c_{Q_2}=s_2s_1s_3$ or $c_{Q_2}=s_2s_3s_1$;
these
are equal because $s_1s_3=s_3s_1$.  \end{example}

\subsection{Definition of $c$-sortable elements}

Let $c=s_{i_1}\dots s_{i_n}$ be a Coxeter element for $W$.  For $J\subseteq\{1,\dots,n\}$, we write $c_J$ for the subword
of $c$ whose reflections are those $s_j$ with $j\in J$.

An element $w\in W$ is called $c$-sortable if it has a reduced expression
$w=c_{J_1}c_{J_2}\dots c_{J_k}$ with $J_1\supseteq J_2 \supseteq \dots \supseteq J_k$.

\begin{example} In the Weyl group of type $A_2$, isomorphic to $S_3$,
  setting $c=s_1s_2$, five of the six group elements are $c$-sortable.
  The one which is not is $s_2s_1$.  In the Weyl group of type $A_3$,
  for each choice of $c$ there are 14 $c$-sortable elements.  In general,
  the $c$-sortable elements in type $A_n$ are counted by the $n$-th
  Catalan number, $\frac1{n+2}{ 2n+2\choose n+1}$.  See \cite[Theorem 9.1]{Re}, which gives complete information about enumeration in all finite types.\end{example}

The $c$-sortable elements of a Coxeter group $W$ were defined by Reading \cite{Re} as a bridge between the combinatorics of finite-type cluster algebras and the
noncrossing partitions of $W$.  They play a central role in the Coxeter-theoretic approach to the combinatorics of cluster algebras, as developed by Reading
and Speyer in a sequence of articles \cite{RS1,RS2,RS3}; they have also
turned out to have interesting applications in the representation theory
of preprojective algebras, see \cite{AIRT}.

There is an extremely useful inductive characterization of
$c$-sortable elements.
In the theorem below, we use the notation
$W_{\la s_{j}\ra}$ for the Coxeter group generated by the reflections $s_i$ for
$i\ne j$.  If $W$ is the Coxeter group corresponding to the quiver
$Q$, then $W_{\la s_{j}\ra}$ is the Coxeter group for the quiver with vertex
$j$ \emph{removed}.  (While initially counter-intuitive, this notation is
standard and very useful.)

\begin{theorem}[{\cite[Lemma 2.4, Lemma 2.5, and Remark 2.7]{Re}}]
  \label{reading}
Let $w\in W$, and let $c=s_{i_1}\dots s_{i_n}$.  
\begin{enumerate}\item
If $\ell(s_{i_1}w)>\ell(w)$, then
$w$ is $c$-sortable iff $w$ is a $s_{i_2}\dots s_{i_n}$-sortable element in $W_{\la s_{i_1}\ra}$.
\item
If $\ell(s_{i_1}w)<\ell(w)$,  then $w$ is $c$-sortable iff $s_{i_1}w$ is $s_{i_2}s_{i_3}\dots s_{i_n}s_{i_1}$-sortable.
\end{enumerate}
\end{theorem}

\begin{proof} Suppose first that $\ell(s_{i_1}w)>\ell(w)$.  This means that
  no reduced expression for $w$ can begin with $s_{i_1}$.  If $w$ is
  $c$-sortable, then, since it does not use the $s_{i_1}$ in the first copy
  of $c$, it can never use any $s_{i_1}$.  Therefore it is in $W_{\la s_{i_1}\ra}$;
  the $s_{i_2}\dots s_{i_n}$-sortability there is immediate; the converse is
  also clear.

  Now suppose that $\ell(s_{i_1}w)<\ell(w)$.  This means that there is a
  reduced expression for $w$ beginning with $s_{i_1}$.  Suppose that $w$ is
  $c$-sortable.  Since $w \not \in W_{\la s_{i_1}\ra}$, any expression for
  $w$ must use $s_{i_1}$ at least once, so the expression which manifests that
  $w$ is $c$-sortable must begin with $s_{i_1}$.  Therefore, it expresses
  $w$ as $s_{i_1}$ times a reduced word for $s_{i_1}w$, which, when we unwind
  the definition, must be $s_{i_2}\dots s_{i_n}s_{i_1}$-sortable.  This is
  exactly what we wanted.   The converse argument works the same way.
  \end{proof}

\begin{example} In type $A_2$, with $c=s_1s_2$,
  the first case of Theorem \ref{reading} applies
  to the $c$-sortable elements $e$ and $s_2$, while the second case 
  applies to $s_1$, $s_1s_2$, and $s_1s_2s_1$,  The non-$c$-sortable element
  $s_2s_1$ falls into the first case of the theorem, which says that it is not
  $c$-sortable because it is not an $s_2$-sortable element of $W_{\la s_1\ra}$
  (obviously true, since it is not an element of $W_{\la s_1\ra}$ at all).  
  \end{example}
  
\subsection{The correspondence between $c$-sortable elements and finite torsion free classes}

A torsion free class in $\rep Q$ is a full additive subcategory
closed under extensions and subrepresentations.
That is to say, if $\mathcal F$ is a torsion free class and $Y$ is in $\mathcal F$, so are all the subrepresentations of $Y$, and if $X,Z\in \mathcal F$,
such that there is an extension
$$0 \rightarrow X \rightarrow Y \rightarrow Z \rightarrow 0,$$
then $Y \in \mathcal F$.

We say that a torsion free class is finite if it contains finitely many
indecomposable representations up to isomorphism.  If $Q$ is a Dynkin quiver,
then all torsion free classes will be finite.  Conversely, if all torsion free classes in $\rep Q$ are finite, then $Q$ is a Dynkin quiver, since the category 
of all representations is itself a torsion free class.


For $w$ a $c_Q$-sortable element of $W$, define $\mathcal F(w)$ to be the
full subcategory of $\rep Q$ consisting of direct sums of copies of the
indecomposable representations whose dimension vectors are roots in $\inv(w)$.  

We can now state the theorem whose proof occupies the remainder of the paper.

\begin{theorem}\label{main} The map $\mathcal F$ is a bijection from $c_Q$-sortable
  elements of $W$ to finite torsion free classes in $\rep Q$.  \end{theorem}

\begin{example} Let $Q$ be the quiver $1\leftarrow 2$. Then
  the correspondence between $c_Q$-sortable elements of $W$ and
  torsion free classes is as follows:
  $$\begin{array}{cc}
    w& \mathcal F(w)\\
    e & 0\\
    s_1 & \add S_1\\
    s_2 & \add S_2\\
    s_1s_2 & \add S_1, P_2\\
    s_1s_2s_1 & \add S_1,P_2,S_2
  \end{array}$$
  Here, we have written $P_2$ for the non-simple indecomposable
  representation, and we have written $\add$ for the additive category
  generated by the subsequent list of indecomposable representations.  
  As is easly verified, the righthand column consists of the torsion free classes of $\rep Q$.  
\end{example}

\subsection{The inductive structure of torsion free classes}
In order to prove Theorem \ref{main}, we will exhibit an inductive structure to 
torsion free classes which parallels the inductive structure of $c$-sortable elements.
Suppose $\mathcal F$ is a torsion free class for $\rep Q$.  Let $i$ be a
sink of $Q$, so $S_i$ is a simple projective.  We want to prove a structure
theorem about $\mathcal F$, parallel to Theorem \ref{reading}.  
We split into two cases, based on whether $S_i \in \mathcal F$.

\begin{proposition} \label{prop1}
  If $S_i\not\in \mathcal F$, then $\mathcal F$ is supported
  on $Q\setminus \{i\}$, and thus defines a torsion free class in the
  representations of $Q \setminus \{i\}$.  
\end{proposition}  

\begin{proof}  
Recall that for $M\in \rep Q$, we have that $\dim\Hom(S_i,M)=\dim M_i$, since $S_i$ is projective.  Thus, if $M_i\ne 0$, then $\Hom(S_i,M)\ne0$.  And note that in this case, the nonzero maps from $S_i$ to $M$ would have
to be injections, because $S_i$ has no non-zero subrepresentations.  That means that
if $M_i\ne 0$, we must have that $M\not\in\mathcal F$, since otherwise the
inclusion of $S_i$ into $M$ would contradict the fact that $S_i\not \in \mathcal F$
but $M\in \mathcal F$.
\end{proof}

\begin{proposition}\label{prop2}
  If $S_i\in \mathcal F$, then $R_i^+ \mathcal F$ is a torsion free class for
  $\mu_i(Q)$ which does not contain $S_i$.  If $\mathcal F$ contains finitely
  many indecomposable representations up to isomorphism, then $\mathcal F'$ contains one fewer.  
\end{proposition}

\begin{proof}
  Let us write $Q'$ for $\mu_i(Q)$.  We will write $X'$, $Y'$, etc. for
  representations of $Q'$ and $X$, $Y$, etc. for representations of $Q$.
  Let $\mathcal F'=R_i^+\mathcal F$.  We begin by proving that $\mathcal F'$
  is closed under subrepresentations.  Let $X'\in \mathcal F'$.  Since $X' \in R_i^+(\rep Q)$, we know that $X'$ has no direct summands isomorphic to $S_i'$.
  Let $Y'$ be a
  subrepresentation of $X'$.

  Let $Y=R_i^-(Y')$ and $X=R_i^-(X')$.  Since $X'\in\mathcal F'$, we know that
  $X\in \mathcal F$.  
  Since $R_i^-$ is only right exact, it is not necessarily the case that
  $Y$ is a subrepresentation of $X$.  However, $R_i^-$ applied to the inclusion
  from $Y'$ to $X'$ gives a map from $Y$ to $X$.  The image of this map is a subrepresentation of $X$, so lies in $\mathcal F$, and the kernel
  of this map is a sum of copies of $S_i$, which also lies in $\mathcal F$.
  Thus, since $\mathcal F$ is extension closed, it also contains $Y$, and thus
  $Y'\in\mathcal F'$, showing that $\mathcal F'$ is closed under subrepresentations.  
  
  We next check that $\mathcal F'$ is closed under extensions.  Suppose that we
  have an extension of $Q'$ representations:
  $$ 0 \rightarrow Y' \rightarrow X' \rightarrow Z' \rightarrow 0$$
  with $Y'$ and $Z'$ in $\mathcal F'$.

  Any indecomposable representation of $Q'$ that admits a nonzero
  morphism from $S_i'$ is isomorphic to $S_i'$.  
  By assumption, neither $Y'$ nor $Z'$ has an $S_i'$ direct summand, so
  $\Hom(S_i',Y')=0=\Hom(S_i',Z')$.  By the $\Hom$ long exact sequence, it
  follows that $\Hom(S_i',X')=0$, so $X'$ does not have any direct summands
  isomorphic to $S_i'$ either.  Applying $R_i^-$, we get
\begin{equation}\label{eqn}
  R_i^-(Y')\rightarrow R_i^-(X')\rightarrow R_i^-(Z')\rightarrow 0,\end{equation}
  since $R_i^-$ is right exact.  
  Since none of $X'$, $Y'$, $Z'$ have any $S_i'$ summands, the dimension
  vector of $R_i^-(X')$ is $s_i(\udim(X'))$, and similarly for $Y'$ and $Z'$.
  Since $s_i$ is a linear map, $s_i\udim(X')=s_i\udim(Y')+s_i\udim(Z')$, and
  it follows that (\ref{eqn}) is exact on the left as well.  Now $R_i^-(Y)$
  and $R^-_i(Z)$ are in $\mathcal F$, so, since (\ref{eqn}) is exact on the left as well,
  we can conclude that $R_i^-(X')$ is also in $\mathcal F$, and $X'\in \mathcal F'$, as desired.  

  The final statement
  follows from the fact that $R_i^+$ sends non-isomorphic
  indecompolsables other than $S_i$ to non-isomorphic indecomposables, while
  killing $S_i$.  
\end{proof}  

The previous two propositions tell us that, given $i$ a sink of $Q$ and $\mathcal F$ a torsion free class in
$\rep Q$, there is a way to associate to it a torsion free class that is ``smaller'' in
a suitable sense: either smaller because it is defined on a smaller quiver
(the case of Proposition \ref{prop1}) or smaller because it contains fewer indecomposables
(the case of Proposition \ref{prop2}).

We now want to reverse the process and show that, given a candidate ``smaller'' torsion free class, it is possible to recover  a
torsion free class in $\rep Q$.

The first case, reversing Proposition \ref{prop1}, is obvious, but we state it
to be able to refer to it later.

\begin{proposition} \label{prop1b} If $\mathcal F$ is a torsion free class in
  the category of representations of $Q\setminus \{i\}$, then it is also a torsion free class in the category of representations of $Q$, via the inclusion
  of $\rep Q\setminus\{i\}$ into $\rep Q$.  
\end{proposition}

The second case, reversing Proposition \ref{prop2}, requires a proof, but it
is similar to that of Proposition \ref{prop2}.

\begin{proposition}\label{prop2b} Let $\mathcal F'$ be a torsion free class in $\rep Q'$. 
  Let $\mathcal F$ be the additive hull of $R_i^-\mathcal F'$ and $S_i$.
  Then $\mathcal F$ is a torsion free class in $\rep Q$.  \end{proposition}

\begin{proof}  
  Let $X$ be in $\mathcal F$, and let $Y$ be a subrepresentation of $X$.  We
  want to show that $Y\in \mathcal F$.  Let $X'=R^+_i(X)$, and let
  $Y'=R^+_i(Y)$.  Since $R^+_i$ kills $S_i$, $R^+_i\mathcal F=R^+_i(R^-_i \mathcal F')\subseteq\mathcal F'$.  
  Since $R^+_i$ is left exact, $Y'$ is a subrepresentation of $X'$, and thus is in $\mathcal F'$.  Now $Y$  is isomorphic to the direct sum of $R_i^+(Y')$ with some
  copies
  of $S_i$, and both of these are in $\mathcal F$, so $Y$ is in $\mathcal F$.  
  
  Now suppose that we have a short exact sequence
  $$ 0 \rightarrow Y \rightarrow X \rightarrow Z \rightarrow 0$$
  with $Y$ and $Z$ in $\mathcal F$.  We want to show that $X\in \mathcal F$.

  Applying $R_i^+$, we get an exact sequence:
  $$ 0 \rightarrow R^+_i(Y) \rightarrow R^+_i(X) \rightarrow R^+_i(Z)$$

  By our assumptions, $R^+_i(Y)$ and $R^+_i(Z)$ are both in $\mathcal F'$.
  In order to get exactness on the right, we can replace $R^+_i(Z)$ by
  the cokernel of the map from $R^+_i(Y)$ to $R^+_i(X)$.  Note that it is
  isomorphic to a subrepresentation of $R^+_i(Z)$, and therefore it is also
  in $\mathcal F'$.  It follows that $R^+_i(X)\in \mathcal F'$, and thus that
  $X$ is in $\mathcal F$.
\end{proof}

Note that we actually only use the above proposition in the case that $\mathcal F'$ does not contain $S_i'$, but we give the more general statement, since the proof is no different.

\subsection{Proof of Theorem \ref{main}}  
We will prove Theorem \ref{main} by establishing the next two propositions.  

\begin{proposition}
  Let $\mathcal F$ be a finite
  torsion free class in $\rep Q$.  Then there is some
  $c$-sortable element $w$ such that $\mathcal F=\mathcal F(w)$.
\end{proposition}

\begin{proof} The proof is by induction on the number of vertices of $Q$ and the number
of indecomposable representations of $\mathcal F$.

Let the first reflection in $c_Q$ be $s_i$.  By the definition of $c_Q$,
vertex $i$ is a sink of $Q$.  
We split into two cases depending on whether or not $S_{i}\in\mathcal F$.

If $S_i\not\in \mathcal F$, then by Proposition \ref{prop1},
$\mathcal F$ is a torsion free class in the representations of
$Q\setminus \{i\}$.  By induction, there is a
$c_{Q\setminus\{i\}}$-sortable element $w$ of $W_{\la i\ra}$ such that
$\mathcal F=\mathcal F(w)$.
This $w$ is also $c_Q$-sortable in $W$ by Theorem \ref{reading}, and we are
done.

If $S_i\in\mathcal F$, then by Proposition \ref{prop2},
$\mathcal F'=R_i^+\mathcal F$ is a
torsion free class for $Q'=\mu_i(Q)$.  By induction, $\mathcal F'=\mathcal F(w')$
for $w'$ a $c_{Q'}$-sortable element.  Since $\mathcal F'$ does not contain $S_i'$, 
we know $\inv(w')$ does not contain $e_i$.  Therefore, no reduced expression for $w'$ begins with $s_i$, and $\ell(s_iw')=\ell(w')+1$.  Let $w=s_iw'$.  Since $c_{Q'}=s_ic_Qs_i$, Theorem
\ref{reading} tells us that $w$ is $c_Q$-sortable.
Now:
\begin{eqnarray*}
\inv(w) &=& \{e_i\} \cup \{s_i\alpha \mid \alpha\in \inv(w')\}\\
&=& \{\udim S_i\} \cup \{ s_i\udim M' \mid M' \textrm{ indecomposable in } \mathcal F'\}\\
&=& \{\udim M \mid M \textrm { indecomposable in } \mathcal F\},
\end{eqnarray*}
and this shows that $w$ is the desired $c$-sortable element of $W$.  
\end{proof}

\begin{proposition} Let $w$ be a $c_Q$-sortable element of $W$. Then
  $\mathcal F(w)$ is a finite torsion free class.\end{proposition}

\begin{proof} The proof is by induction on the number of vertices of $Q$
  and the number of inversions of $w$.  Let the first reflection of $c_Q$
  be $s_i$. Then vertex $i$ is a sink of $Q$.
  We split into two cases depending on whether $\ell(s_iw)>\ell(w)$ or
  $\ell(s_iw)<\ell(w)$.

  Suppose first that $\ell(s_iw)>\ell(w)$.  Since $w$ is $c_Q$-sortable,
  Theorem \ref{reading} tells us that $w$ is $c_{Q\setminus i}$-sortable in
  $W_{\la i\ra}$.  By induction, $\mathcal F(w)$ is a torsion free
  class in $\rep Q\setminus\{i\}$.  Now Proposition \ref{prop1b} tells
  us that $F(w)$ is also a torsion free class in $\rep Q$.

  Suppose next that $\ell(s_iw)<\ell(w)$.  Theorem \ref{reading} tells
  us that $s_iw$ is $s_ic_{Q}s_i$-sortable. We note that this means that
  $s_iw$ is $c_{\mu_i(Q)}$-sortable, so, by induction, $\mathcal F(s_iw)$ is
  a torsion free class for $\mu_i(Q)$.  We know that
  \begin{eqnarray*}
    \inv(w)&=&\{e_i\}\cup
        \{s_i\alpha \mid \alpha\in \inv(s_iw)\}\\
        &=& \{\udim S_i\} \cup \{ \udim R_i^-M \mid M \textrm { indecomposable in }\mathcal F(s_iw)\}\end{eqnarray*}
  Proposition \ref{prop2b} now
  tells us that $\mathcal F(w)$ is a torsion free class, as desired.  
\end{proof}

This completes the proof of Theorem \ref{main}.

\end{document}